\documentclass[11pt]{amsart}
\usepackage{amsfonts}
\usepackage{amsfonts,latexsym,rawfonts,amsmath,amssymb,amsthm, a4,a4wide}

\usepackage[plainpages=false]{hyperref}

\usepackage{graphicx}

\RequirePackage{color}

\numberwithin{equation}{section}

\newcommand{\beq}{\begin{equation}}
\newcommand{\eeq}{\end{equation}}
\newcommand{\beqs}{\begin{eqnarray*}}
\newcommand{\eeqs}{\end{eqnarray*}}
\newcommand{\beqn}{\begin{eqnarray}}
\newcommand{\eeqn}{\end{eqnarray}}
\newcommand{\beqa}{\begin{array}}
\newcommand{\eeqa}{\end{array}}

\newcommand{\R}{\mathbb R}
\newcommand{\calE}{{\mathcal E}}

\newcommand{\abs}[1]{\left\vert#1\right\vert}

\newcommand{\bh}{{\bar h}}

\newcommand{\e}{\varepsilon}
\newcommand{\p}{\partial}

\newcommand{\Om}{\Omega}

\newcommand{\diam}{\mbox{diam}\,}
\newcommand{\comment}[1]{}
\def\h{\hspace*{.24in}} 

{\begin{list}{}%
         {\setlength{\leftmargin}{#1}}%
         \item[]%
}
{\end{list}}

\newtheorem{prop}{Proposition}[section]
\newtheorem{thm}[prop]{Theorem}
\newtheorem{lem}[prop]{Lemma}

\newtheorem{cor}[prop]{Corollary}

\theoremstyle{remark}
\newtheorem{rem}[prop]{Remark}

\newcommand{\dist}{\text{dist}}

\author{Nam Q. Le}
\address{Department of Mathematics, Indiana University, 
Bloomington, IN 47405, USA. }
\email {nqle@iu.edu}
\author{Ovidiu Savin}
\address{Department of Mathematics, Columbia University, New York, NY 10027, USA}
\email{savin@math.columbia.edu}

\thanks{N. Q. L. was supported by NSF grants DMS-2054686 and DMS-2452320.  O. S. was supported by NSF grant DMS-2349794}
\allowdisplaybreaks
\title[Global $C^{1,\beta}$ and $W^{2, p}$ regularity for singular Monge--Amp\`ere]{Global $C^{1,\beta}$ and $W^{2, p}$ regularity for some singular Monge--Amp\`ere equations}

\makeatletter
\@namedef{subjclassname@2020}{%
  \textup{2020} Mathematics Subject Classification}
\makeatother

\begin{document}
\begin{abstract}
We establish global $C^{1,\beta}$ and $W^{2, p}$ regularity for  singular Monge--Amp\`ere equations of the form
\[\det D^2 u \sim \text{dist}^{-\alpha}(\cdot,\partial\Omega),\quad \alpha\in (0, 1),\]
under suitable conditions on the boundary data and domains. 
Our results imply that
the convex Aleksandrov solution to the singular Monge--Amp\`ere equation 
\[\det D^2 u=|u|^{-\alpha}\quad \text{in}\quad\Omega,\quad u=0\quad \text{in}\quad \partial\Omega, \quad \alpha\in (0, 1),\]
where $\Omega$ is a $C^3$, bounded, and uniformly convex domain, 
is globally $C^{1,\beta}$ and belongs to $W^{2, p}$ for all $p<1/\alpha$.

\end{abstract}

\maketitle

\section{Introduction and statement of the main result}
\label{Sect1}
In this paper, we are interested in establishing global H\"older gradient estimates and global Sobolev estimates for the second derivatives of solutions to certain singular Monge-Amp\`ere equations. Before stating our results, we first briefly recall closely related estimates in nondegenerate and degenerate equations.

\smallskip
Global $C^2$ estimates for the nondegenerate Monge--Amp\`ere equation
\begin{equation}
\label{MAf}
\det D^2 u= f \quad \text{in}~\Omega,
\quad 0<\lambda\leq f\leq \Lambda,
\quad u = \varphi\quad \text{ on}~ \p\Omega\end{equation}
were first established the works of Ivo\u{c}kina \cite{I}, Krylov \cite{K}, Caffarelli-Nirenberg-Spruck \cite{CNS} when $f\in C^2(\overline{\Omega})$.
For $f\in C^\alpha(\overline{\Omega})$,
 under sharp conditions
on the boundary data, global $C^{2,\alpha}$ estimates were obtained by Trudinger-Wang \cite{TW1} and the second author \cite{SC2a}. 
For $f\in C(\overline{\Omega})$, the second author \cite{Sw2p} established global $W^{2, p}$ estimates for solutions to \eqref{MAf} under suitable conditions on the boundary data and domain. Without any continuity on $f$, the techniques in \cite{Sw2p} give global $W^{2, 1+\e}$ estimates for solution to \eqref{MAf} where $\e>0$ depends on $n,\lambda$, and $\Lambda$.

\smallskip
When $f$ is only bounded between two positive constants, global $C^{1,\beta}$  estimates for \eqref{MAf} were established by the authors in \cite{LS1} for $C^3$, uniformly convex domains,  and then by the second author and Zhang \cite{SZ1} under optimal boundary conditions when the domain is uniformly convex or has flat boundary. 
Recently, Caffarelli, Tang, and Wang \cite{CTW} established these estimates when the Monge--Amp\`ere measure $f$ is doubling and bounded from above, which allows for degeneracy and the case of zero right-hand side.

\smallskip
For degenerate Monge-Amp\`ere equations of the type
\begin{equation}
\label{MAa}
\det D^2 u \sim \dist^{\alpha}(\cdot,\p\Omega),\quad \alpha>0,
\end{equation}
based on the boundary localization theorem in \cite{SC2}, the authors \cite{LS2} established 
global $C^{2,\beta}(\overline{\Omega})$ estimates for solutions under suitable conditions on the boundary data and domain. In particular, these estimates imply global $C^{2,\beta}(\overline{\Omega})$ regularity for convex solutions to degenerate Monge-Amp\`ere equations
\begin{equation}
\label{MAevp}
\det D^2 u=|u|^\alpha \quad\text{in }\Omega, \quad u=0 \quad\text{on }\p\Omega,\quad \alpha>0,\end{equation}
on $C^3$, uniformly convex domains $\Omega$. The case of $\alpha=n$ and $|u|^\alpha$ being replaced by $\lambda (\Omega) |u|^n$ is the Monge-Amp\`ere eigenvalue problem.

\smallskip
Here, we investigate higher-order regularity for a class of singular Monge-Amp\`ere equations where $|u|^{\alpha}$ in \eqref{MAevp} is replaced by  $|u|^{-\alpha}$. This type of equations has 
a close relation with the 
$L_p$-Minkowski problem (see,  for example \cite{Luk}) and the Minkowski problem in centro-affine geometry (see, for example \cite{CW, JW} and the references therein).
When $\alpha\in (0, 1)$, our main result 
establishes 
global $C^{1,\beta}$ and $W^{2, p}$ regularity. 
\begin{thm}[Global $C^{1,\beta}$ and $W^{2, p}$ regularity for singular Monge-Amp\`ere equations] 
\label{mainthm1}
Let $\alpha\in (0, 1)$ and let $\Omega\subset\R^n$ ($n\geq 2$) be a uniformly convex domain with $C^3$ boundary. Let $u\in C(\overline{\Omega})$ be the Aleksandrov convex solution to
the singular Monge-Amp\`ere equation
\begin{equation}
 \label{MA1}
 \left\{
 \begin{alignedat}{2}
   \det D^{2} u~&= |u|^{-\alpha} \h~&&\text{in} ~\Omega, \\\
u &=0\h~&&\text{on}~\p \Omega.
 \end{alignedat}
 \right.
\end{equation}
Then, the following hold.
\begin{enumerate}
\item[(i)] $u\in C^{1,\beta}(\overline{\Omega})$ for some $\beta=\beta(n,\alpha,\Omega)\in (0, 1)$. 
\item[(ii)] $u\in W^{2, p}(\Omega)$ for all $p<1/\alpha$.
\end{enumerate}
\end{thm}

\smallskip
For $\alpha\in (0, 1)$, Mohammed \cite[Corollaries 2.2 and 3.4]{M} proved the existence of a convex solution $u\in C^{\infty}(\Omega)\cap C^{0,1}(\overline{\Omega})$ to \eqref{MA1}. 
Thus, there exist positive constants $\lambda$ and $\Lambda$ depending on $n, \alpha$, and $\Omega$ such that
\begin{equation}
\label{singbdr1}
 \lambda [\dist(\cdot,\p\Omega)]^{-\alpha}\le \det D^2 u \le\Lambda [\dist(\cdot,\p\Omega)]^{-\alpha}\quad\mathrm{in}\;\Omega.
\end{equation}
It can be shown using a maximum principle argument that $u$ is unique.  Moreover, $u$ separates quadratically from its tangent hyperplanes on the boundary. This follows from the proof of Proposition 3.2 in \cite{SC2a} where only the lower bound for $\det D^2 u$ is used. This condition clearly follows from \eqref{singbdr1}.

\smallskip
On the other hand, for $\alpha>1$, it can be showed that for the solution $u$ to \eqref{MA1}, $|u|$ is comparable to $[\dist(\cdot,\p\Omega)]^{\frac{n+1}{n+\alpha}}$, so it is only H\"older continuous but not Lipschitz continuous; see Lazer and Mckenna \cite[Theorem 3.1]{LM}.

\smallskip
We say a few words about the proof of Theorem \ref{mainthm1} which follows from general results for singular Monge-Amp\`ere equations satisfying \eqref{singbdr1}; see Theorems \ref{thmb} and \ref{thmw}.
Our main technical tool
is a boundary localization theorem (Theorem \ref{SZthm}) for singular Monge-Amp\`ere equations of the type \eqref{singbdr1}. As in \cite{Sw2p}, Theorem \ref{SZthm} allows us to control the geometry of maximal interior sections of the solution in Proposition \ref{tan_sec0}. From this,  using Caffarelli's interior $C^{1,\beta}$ estimate together with a rescaling argument, we obtain the global $C^{1,\beta}$ estimates (Theorem \ref{thmb}). We use the boundary H\"older gradient estimates to prove a Vitali-type covering lemma (Lemma \ref{V_lem}). We deduce the global $W^{2,p}$ estimates (Theorem \ref{thmw}) from Caffarelli's interior $W^{2,p}$ estimate in combination with a covering argument. 

\smallskip
The paper is organized as follows. In Section \ref{SectBLT}, we recall the Boundary Localization Theorem and use it to study the geometry of sections for singular Monge-Amp\`ere equations. We prove Theorem \ref{mainthm1}(i) in Section \ref{SectC1b}. In Section \ref{SectV}, we prove a Vitali Covering Lemma for sections near the boundary. In the final section, Section \ref{SectW2p}, we prove Theorem \ref{mainthm1}(ii). 
\section{Boundary localization theorem and geometry of sections}
\label{SectBLT}
 Throughout, we assume that $u\in C^{1}(\Omega)\cap C^{0,1}(\overline{\Omega})$ is a convex function in a bounded domain $\Omega$ in $\R^n$. We begin with some notation. We usually write $x=(x', x_n)$ for $x\in\R^n$. 
For $x_0\in\p\Omega$, let $\nu_{x_0}$ be the outer unit normal to $\partial\Omega$ at $x_0$. Choose $\tau_{x_0}\in (\R^n)^{n-1}$ in the supporting hyperplane to $\p\Omega$ at $x_0$ such that $(\tau_{x_0},\nu_{x_0})$ is an orthogonal coordinate frame in $\R^n$.

Define the section of $u$ with center $x_0\in\overline{\Omega}$ and height $h>0$ by
\[S_u(x_0, h):=\{x\in \overline{\Omega}: u(x)< u(x_0) + Du(x_0)\cdot (x-x_0) + h\}.\]
In the above definition and the paper, when $x_0\in\p\Omega$, $Du(x_0)$ is understood as follows: $$ x_{n+1}=u(x_0)+Du(x_0) \cdot (x-x_0) $$ is a supporting 
hyperplane for the graph of $u$ in $\overline{\Omega}$ at $(x_0, u(x_0))$, but for any $\e >0$,
$ x_{n+1}=u(x_0)+(Du(x_0)- \e \nu_{x_0}) \cdot (x-x_0)$
is not a supporting hyperplane.

For  $x\in \Omega$, we denote by $\bar{h}(x)$ the maximal height of all sections of $u$ centered at $x$ and contained in $\Omega$, that is,
$$\bar{h}(x): =\sup\{h> 0: \quad S_{u}(x, h)\subset \Omega\}.$$
In this case, $S_{u}(x, \bar{h}(x))$  is called the {\it maximal interior section of $u$ with center $x\in\Omega$}, and it is tangent to the boundary $\p\Omega$ at some point $z$, that is, 
$\p S_{u}(x, \bar h(x))\cap \p\Omega =\{z\}$.

We fix $\alpha\in (0, 1)$. Let \[\calE:=\{x\in\R^n: |x'|^2 + x_n^{2-\alpha}<1\}.\]
Note that \[B_{1/2}(0)\subset\calE\subset B_1(0).\]
Let us denote 
\[\Omega_h:=\{x\in\Omega: \dist(x,\p\Omega)<h\},\quad 
A_h = (h^{\frac{1}{2}},\ldots, h^{\frac{1}{2}}, h^{\frac{1}{2-\alpha}})\quad\text{for } h>0.\]

We recall the following volume estimates for sections of convex functions (see \cite[Lemmas ~ 5.6 and 5.8]{L}).
\begin{lem}[Volume estimates] 
\label{volsec}
Let $u\in C^{1}(\Omega)\cap C^{0,1}(\overline{\Omega})$ be a convex function in a bounded domain $\Omega$ in $\R^n$. 
\begin{itemize}
\item If $\det D^2 u\leq \Lambda$ in $S_u(x_0, h)\Subset\Omega$, 
then 
\[|S_u(x_0, h)|\geq c(n) \Lambda^{-1/2} h^{n/2}.\]
\item If $\det D^2 u\geq \lambda>0$ in $S_u(x_0, h)$, then 
\[|S_u(x_0, h)|\leq C(n) \lambda^{-1/2} h^{n/2}.\]
\end{itemize}
\end{lem}
\subsection{Boundary Localization Theorem for singular equations} In \cite{SZ2}, the second author and Zhang established
the following boundary localization theorem which is a singular counterpart of the boundary localization theorem in \cite{SC2a}.
\begin{thm}[Boundary Localization Theorem for singular equations]
\label{SZthm}
Let $\Omega$ be a bounded convex domain in $\R^n$ with $C^2$ boundary $\partial\Omega$. Let $\alpha\in (0, 1)$ and $0<\lambda\leq\Lambda$. Assume $u\in C(\overline{\Omega})$ is a convex function satisfying
\begin{equation*}
 \lambda [\emph{dist}(\cdot,\p\Omega)]^{-\alpha}\le \det D^2 u \le\Lambda [\emph{dist}(\cdot,\p\Omega)]^{-\alpha}\quad\mathrm{in}\;\Omega,
\end{equation*}
 and on $\partial\Omega$, $u$ separates quadratically from its tangent hyperplane, namely,  
there exists $\mu>0$ such that for all $ x_0, x\in \p\Omega$, 
\[\mu|x-x_0|^2\le u(x)-u(x_0)-D u(x_0)\cdot(x-x_0)\le\mu^{-1}|x-x_0|^2.\]
Then, there is a constant $c>0$ depending only on $n$, $\lambda$, $\Lambda$, $\alpha$, $\mu$, 
$\mathrm{diam}(\Omega)$, and the  $C^2$ regularity of $\p\Omega$, 
such that for each $x_0\in\partial\Omega$ and $h\le c$, 
\begin{equation*}
\mathcal{E}_{ch}(x_0)\cap\overline{\Omega}\subset S_u(x_0, h)\subset\mathcal{E}_{c^{-1}h}(x_0),
\end{equation*}
where
$$\mathcal{E}_h(x_0):=\big\{x\in\R^n: |(x-x_0)\cdot \tau_{x_0}|^2+|(x-x_0)\cdot\nu_{x_0}|^{2-\alpha}<h\big\}.$$
\end{thm}

\begin{rem} 
One can check the quadratic separation hypotheses  of Theorem \ref{SZthm} in several scenarios concerning the boundary data, such as
\begin{enumerate}
\item when $u|_{\p\Omega}=0$ and $\Omega$ is uniformly convex
\item when $u|_{\p\Omega}\in C^3$, $\p\Omega\in C^3$ and $\Omega$ is uniformly convex.
\end{enumerate}
 In each of these cases, the quadratic separation follows from the proof of Proposition 3.2 in \cite{SC2a} where only the lower bound for $\det D^2 u$ is used. This condition clearly follows from $\det D^2 u\geq \lambda [\dist(\cdot,\p\Omega)]^{-\alpha}$. 
 \end{rem}
 \begin{rem}
 With some more computations, one can also reduce the above $C^3$ regularity of both $u|_{\p\Omega}$ and $\p\Omega$ to $C^{2, 1-\gamma}$ where $0<\gamma< 3\alpha/(4-\alpha)$. We sketch the argument at a point $x_0\in\p\Omega$ as follows. 
 
 By changing coordinates and subtracting an affine function from $u$ and $\varphi=u|_{\p\Omega}$, we can assume that $\Omega\subset\R^n_{+}=\{x\in\R^n: x_n>0\}$, $x_0=0\in\p\Omega$,
$u(0)=0$, and $Du(0)=0$. Then $\varphi\geq 0$. Since $\varphi,\p\Omega\in C^{2, 1-\gamma}$, we find that
\[\varphi(x) = \sum_{i<n} \frac {\mu_i}{2} x_i^2  +  O(|x'|^{3-\gamma}), \quad \mbox{with} \quad \mu_i \ge 0.\]
We need to show that $\mu_i>0$ for all $i=1, \ldots, n-1$. 

Assume $\mu_1=0$. Now, if we restrict to $\p \Omega$ in a small neighborhood near the origin, then for 
all small $h$ the set $\{\varphi<h \}$ contains $ \{|x_1| \le c_1 h^{1/(3-\gamma)}\} \cap \{ |x'| \le c_1 h^{1/2} \} $
for some $c_1>0$.

 Since $S_h=S_u(0, h):=\{x\in\overline{\Omega}: u(x)<h\}$ contains the convex set 
generated by $\{\varphi <h\}$, and $x_n\geq c_2|x'|^2$ in $S_h(0, h)$ because $\Omega$ is uniformly convex, we have
$$|S_u(0, h)| \ge c_2 (c_1 h^{1/(3-\gamma)})^3 c_1^{n-2}h^{(n-2)/2} =c_3 h^{\frac{\gamma}{3-\gamma}}h^{n/2}.$$
Let $x_h^\ast$ be the center of mass of $S_h$ and $d_h:= x_h^\ast \cdot e_n$. From John's lemma, we have \[S_h\subset \{x_n\leq C(n) d_h\}.\] 
Thus, $\det D^2 u\geq \lambda (C(n) d_h)^{-\alpha}$ in $S_h$, and 
Lemma \ref{volsec} gives 
\[|S_h| \le C(n) \Big(\lambda (C(n) d_h)^{-\alpha}\Big)^{-1/2}h^{n/2}= C(n,\lambda, \alpha) d_h^{\alpha/2} h^{n/2}.\] 
On the other hand, we have
\[|S_h| \geq c(n) (c_1 h^{1/2})^{n-1} d_h.\]
The last two estimates on $|S_h|$ implies that 
\[d_h\leq C h^{\frac{1}{2-\alpha}},\]
where $C$ is independent of $h$. It follows that
\[c_3 h^{\frac{\gamma}{3-\gamma}}h^{n/2} \leq |S_h| \leq C(n,\lambda, \alpha) d_h^{\alpha/2} h^{n/2} \leq C h^{\frac{\alpha}{2(2-\alpha)}} h^{\frac{n}{2}}.\]
However,   from  $0<\gamma< 3\alpha/(4-\alpha)$, we have $\frac{\gamma}{3-\gamma} <\frac{\alpha}{2(2-\alpha)}$ which easily gives a contradiction to the preceding inequality as $ h \to 0$. Therefore, we must have $\mu_i>0$ for all $i$.
\end{rem}
\begin{rem} Unless otherwise stated, positive constants depending only on $n$, $\lambda$, $\Lambda$, $\alpha$, $\mu$, and $\Omega$ (via
$\mathrm{diam}(\Omega)$ and the  $C^2$ regularity of $\p\Omega$) are  called {\it universal}.  They are usually denoted by $c, c_\ast, c_0, c_1$, $C, C_0, C_1, C^\ast$, etc., where the lowercase letters indicate small constants and uppercase letters indicate large constants.
\end{rem}

Observe that 
the function
$u$ in Theorem \ref{SZthm} is differentiable at each $x_0\in\p\Omega$ and $Du(x_0)$ is in fact the classical gradient of $u$ at $x_0$.
\begin{prop}[Pointwise $C^{1, 1-\alpha}$ estimates at the boundary]
\label{U-bound-grad2}
Assume that $u$ and $\Omega$ satisfy the hypotheses of  Theorem \ref{SZthm} at a point $x_0 \in \p \Omega$. Then $u$ is differentiable at $x_0$, and for $x\in \overline \Omega \cap B_r(x_0)$ where $r \leq c(n,\lambda,\Lambda,\alpha,\mu, \Omega)$, we have
\begin{equation}
\label{U-bound-grad}
C^{-1}|x-x_0|^2\leq  u(x)-u(x_{0})-Du(x_{0})\cdot (x-x_{0}) \leq C\abs{x-x_{0}}^{2-\alpha},
\end{equation}
where $C=C(n,\lambda,\Lambda,\alpha,\mu, \Omega)$.
Moreover, if $u$ and $\Omega$ satisfy the hypotheses of Theorem \ref{SZthm} also at another point $z_0 \in \p \Omega \cap B_r(x_0)$, then
$$|Du(z_0)-Du(x_0)| \le Cr^{1-\alpha}.$$
\end{prop}
\begin{proof}
We can assume that $\Omega\subset\R^n_{+}=\{x\in\R^n: x_n>0\}$, $x_0=0\in\p\Omega$,
$u(0)=0$, and $Du(0)=0$. Note that  $u\geq 0$ in $\Omega$.
For $h \le c_0 (n,\lambda,\Lambda,\alpha,\mu, \Omega)$, Theorem \ref{SZthm} asserts that
$$ \calE_{c_0h} \cap \bar \Omega  \subset S_u(0, h) \subset \calE_{c_0^{-1}h}\cap \bar \Omega,\quad \text{where } \calE_h:=\{x\in\R^n: |x'|^2 + x_n^{2-\alpha}<h\}= A_h\calE.$$ 
It follows that, for $c$ and $C $ depending only on $\Omega$, $\alpha,\mu $, $\lambda$, $\Lambda$, and $n$, we have
\begin{equation}\label{small-sec0}
 \overline \Omega \cap B_{ch^{\frac{1}{2-\alpha}}}(0)\subset S_{u}(0, h)   \subset
  \overline \Omega\cap B_{C h^{1/2}}(0).
\end{equation}
The first inclusion of \eqref{small-sec0} gives  $|u| \le h$ in $\overline \Omega \cap B_{ch^{\frac{1}{2-\alpha}}}(0)$.
Thus, for all $x$ close to the origin 
\begin{equation}
\label{ubdlog}
|u(x)| \le C |x|^{2-\alpha},
\end{equation}
which shows that $u$ is differentiable 
at $0$. 
The other inclusion of \eqref{small-sec0} gives a lower bound for $u$ near the origin
\[
u(x) \ge C^{-1} |x|^2.
\]
Therefore, \eqref{U-bound-grad} is proved.

Suppose now the hypotheses of Theorem \ref{SZthm} are satisfied at $z_0 \in \p \Omega \cap B_r(0)$. We need to show
 that for $\hat z:= Du(z_0)$, $$|\hat z|=|Du(z_0)|\leq C r^{1-\alpha}.$$ For this, we use \eqref{U-bound-grad} for $z_0$ at $0$, and for $0$ and $z_0$ at all points $x$ in a ball $$B:=\overline{B_{c(n,\Omega)r}(y)} \subset \Omega\cap B_r(0)\cap B_{r}(z_0).$$ 
For $x\in B$, we have
\begin{eqnarray*}
Du(z_0)\cdot x& \leq& u(x)+ [- u(z_0) + Du(z_0)\cdot z_0]-C^{-1}|x-z_0|^2\\&\leq&
C|x|^{2-\alpha} + [C|z_0|^{2-\alpha}- u(0)]\\&\leq& 2Cr^{2-\alpha}.
\end{eqnarray*}
The lower bound for $Du(z_0)\cdot x$ is obtained similarly, and we have \[|\hat z\cdot x|=|Du(z_0)\cdot x|\leq C_0r^{2-\alpha}\quad\text{for all }x\in B.\]
We use the above inequality at $y$ and $\hat y:=y + cr \hat z/|\hat z|$ (if $\hat z\neq 0$) to get
\[cr|\hat z|= \hat z \cdot \hat y- \hat z\cdot y\leq 2C_0r^{2-\alpha}.\]
Therefore $|\hat z|\leq (2C_0/c) r^{1-\alpha}$, as desired.
\end{proof}

\subsection{Geometry of maximal interior sections}
Below, we
summarize key geometric properties of maximal interior sections.
\begin{prop}[Shape of maximal interior sections] 
\label{tan_sec0}
Let $u$ and $\Omega$ satisfy the hypotheses of Theorem \ref{SZthm}. 
Assume that for some $y \in \Omega$, the maximal interior section $S_{u}(y, \bar h(y)) \subset \Omega$
is tangent to $\p \Omega$ at $x_0$.
 If  $h:=\bar h(y) \leq c_\ast$ where $c_\ast$ is a small universal constant,  then 
there exists a small 
 positive universal constant $\kappa_0$ such that
 \begin{equation*}
   \left\{\begin{alignedat}{1}
   &Du(y)-Du(x_0)= -a \nu_{x_0}
\quad \mbox{for some} \quad   a \in [\kappa_0 h^{\frac{1-\alpha}{2-\alpha}}, \kappa_0^{-1} h^{\frac{1-\alpha}{2-\alpha}}],\\
&\kappa_0 \calE_h(x_0) \subset S_{u}(y, h) -y\subset \kappa_0^{-1} \calE_h(x_0), \,\text{and}\\
& \kappa_0 h^{\frac{1}{2-\alpha}} \le \emph{dist}(z,\p \Omega) \le \kappa_0^{-1} h^{\frac{1}{2-\alpha}}\quad\text{for all } z\in S_u(y, 3h/4). &&
   \end{alignedat}\right.
   \end{equation*}
\end{prop}

\begin{proof} For simplicity, we can assume $\Omega\subset\R^n_{+}=\{x\in\R^n: x_n>0\}$, $x_0=0\in\p\Omega$,
$u(0)=0$, and $Du(0)=0$. Note that $\nu_{x_0}=-e_n$, and $u\geq 0$ in $\Omega$. Denote $\calE_h= \calE_h(0)$. Consider $c_\ast$ to be not greater than the constant $c$ in Theorem \ref{SZthm}. Assume $h=\bar h(y)\leq c_\ast$.

Because the section $$S_u(y,h)=\{x\in\overline{\Omega}: u(x) < u(y) + Du(y)\cdot (x-y) + h\}\subset \Omega$$
is tangent to $\p \Omega$ at $0$, we must have
$$u(0)= u(y) + Du(y)\cdot (0-y) + h,\quad
\text{and } Du(0) - Du(y) = -a e_n$$
for some $a\in\R$. Since $u(0)=0$ and $Du(0) =0$, we have 
$$Du(y) = a e_n,\quad u(y) + h= a y_n, \quad\text{and }S_u(y, h)=\{x\in \overline{\Omega}: u(x) < a x_n\}.$$
The same arguments show that
\begin{equation}\label{section-expression}S_u(y, t) =\{x\in \overline{\Omega}: u(x) < a x_n + t-h\}\quad \text{for all } t>0. 
\end{equation}
For $t> 0$, we denote
$$S_t':=\{ x \in \overline \Omega: u(x)<tx_n\} ,$$
and clearly $S_{t_1}'\subset S_{t_2}'$ if $t_1 \le t_2$.

 \smallskip
\noindent
{\bf Step 1.} We show that 
\begin{equation}\label{f_sub}
S_{c_1 h^{\frac{1-\alpha}{2-\alpha}}}' \subset S_u(0, h) \cap \Omega\quad \text{for } c_1:= c^{\frac{1}{2-\alpha}}.
\end{equation}
Indeed, if (\ref{f_sub}) does not hold, then from $u(0)=0$ and the convexity of $u$, we can find \[x\in S_{c_1 h^{\frac{1-\alpha}{2-\alpha}}}'\cap \p S_u(0, h)\cap\Omega.\] 
By Theorem \ref{SZthm}, $x\in \overline{\calE_{c^{-1}h}}$.
Thus $u(x)=h$ and $x_n^{2-\alpha}\leq c^{-1} h$, so
\[x_n \leq (c^{-1} h)^{\frac{1}{2-\alpha}}.\]
With $c_1$ defined as above, 
we now have 
\[
h= u(x) < c_1 h^{\frac{1-\alpha}{2-\alpha}} x_n\leq c_1 h^{\frac{1-\alpha}{2-\alpha}} (c^{-1} h)^{\frac{1}{2-\alpha}}=h,
\]
which is a contradiction.   Hence, (\ref{f_sub}) holds. 

 \smallskip
\noindent
{\bf Step 2.} We next show that
\[a\geq c_1 h^{\frac{1-\alpha}{2-\alpha}}. \]
 If this is not true, then $a< c_1 h^{\frac{1-\alpha}{2-\alpha}}$, so
 \[y_n=\frac{u(y) + h}{a}\geq \frac{h}{a} >c^{-1}_1 h^{\frac{1}{2-\alpha}}.\]
In view of Step 1, we have $y\in S'_a \subset S_{c_1 h^{\frac{1-\alpha}{2-\alpha}}}'\subset S_u(0, h) \cap \Omega$. Hence, Theorem \ref{SZthm} gives
\[y_n \leq (c^{-1} h)^{\frac{1}{2-\alpha}} =c^{-1}_1 h^{\frac{1}{2-\alpha}},\]
which contradicts the preceding estimate.

 \smallskip
\noindent
{\bf Step 3.}  We show that $S_{c_1 h^{\frac{1-\alpha}{2-\alpha}}}'$ has volume comparable to that of $S_u(0, h)$. 

By \eqref{f_sub}, we only need to prove the lower bound. 
Let
\[\tilde u(x):= u(x)- c_1 h^{\frac{1-\alpha}{2-\alpha}} x_n,\quad\text{and}\quad\theta:= c^{\frac{2}{1-\alpha}}/2^{\frac{2-\alpha}{1-\alpha}}.\]
From Theorem \ref{SZthm}, there exists $\bar x\in S_u(0, \theta h)$ such that $\bar x_n\geq (c\theta h)^{\frac{1}{2-\alpha}}$.  Note that
\[\tilde u (\bar x) \leq \theta h - c_1 h^{\frac{1-\alpha}{2-\alpha}} (c\theta h)^{\frac{1}{2-\alpha}} =-\theta h.\]
Let $x_0$ be the minimum point of $\tilde u$ in $\overline{\Omega}$. Then \[x_0\in\Omega, \quad \tilde u(x_0)\leq -\theta h,\quad \text{and } 
S_{c_1 h^{\frac{1-\alpha}{2-\alpha}}}' = S_{\tilde u} (x_0, -\tilde u(x_0)).\]
Let us consider $z\in\Omega$ with

\[\tilde u (z)  \leq -\theta h/2.\]
Then from 
\[ -\theta h/2 \geq -c_1  h^{\frac{1-\alpha}{2-\alpha}} z_n,\]
we find
\[z_n\geq \tilde c h^{\frac{1}{2-\alpha}},\quad \tilde c:= \theta/(2c_1).\]

We prove that for some universal constant $\hat c$,
\begin{equation}
\label{distx1}
\hat c h^{\frac{1}{2-\alpha}} \leq \dist(z,\p\Omega) \leq (c^{-1 } h)^{\frac{1}{2-\alpha}}. \end{equation}
Indeed, since $z\in S_{c_1 h^{\frac{1-\alpha}{2-\alpha}}}'$, we deduce from Step 1 and Theorem \ref{SZthm} that
\[|z'|\leq (c^{-1} h)^{\frac{1}{2}},\quad z_n\leq  (c^{-1 } h)^{\frac{1}{2-\alpha}}.\]
Thus, the second inequality in \eqref{distx1} is obvious. 
For the first inequality, observe that
\[\p\Omega\cap B_{4c^{-1}h}(0)= \{(x', g(x'))\},\]
 where 
 \[0\leq g(x')\leq C|x'|^2 \leq C c^{-1} h,\]
 for some universal constant $C$. 
 We assert that
\begin{equation}
\label{distx2} \dist(z,\p\Omega) \leq z_n -g(z') \leq 2  \dist(z,\p\Omega).
\end{equation}
Indeed, let $\bar z\in\p\Omega$ be such that $\dist(z,\p\Omega)=|z-\bar z|$. Then
\[|\bar z| \leq |\bar z- z| + |z|\leq 2|z|.\]
It follows that
\[\dist(z,\p\Omega)\geq z_n-g(\bar z')\geq z_n -C|z'|^2 \geq z_n -4C|z|^2 \geq z_n/2,\]
if $h$ is small.  The  second inequality in \eqref{distx2} follows. 

We have
\[z_n-g(z') \geq \tilde c h^{\frac{1}{2-\alpha}} - Cc^{-1} h\geq  \frac{\tilde c}{2} h^{\frac{1}{2-\alpha}},\]
provided that
\[h\leq h_0,\]
where $h_0$ is small, universal.

It follows from \eqref{distx2} that 
\[ \dist(z,\p\Omega) \geq  \frac{\tilde c}{4} h^{\frac{1}{2-\alpha}}\equiv \hat ch^{\frac{1}{2-\alpha}}. \]
Since $z$ is arbitrary, we deduce from \eqref{distx2} that
\[\det D^2 u \leq \Lambda \dist^{-\alpha}(\cdot,\p\Omega) \leq \Lambda (\hat ch^{\frac{1}{2-\alpha}})^{-\alpha}\quad\text{in } S_{\tilde u} (x_0, -\tilde u(x_0)-\theta h/2).\]
Thus, the volume estimate in Lemma \ref{volsec}(i) gives
\begin{eqnarray*}
|S_{c_1 h^{\frac{1-\alpha}{2-\alpha}}}'| &\geq& |S_{\tilde u} (x_0, -\tilde u(x_0)-\theta h/2)| \\
&\geq & c(n) \big[ \Lambda (\hat ch^{\frac{1}{2-\alpha}})^{-\alpha}\big]^{-1/2} |-\tilde u(x_0)-\theta h/2|^{n/2}\\
&\geq& c' h^{n/2} h^{\frac{\alpha}{2(2-\alpha)}}=c' h^{\frac{n-1}{2}} h^{\frac{1}{2-\alpha}} \\
&\geq& c'' |S_u(0, h)|,
\end{eqnarray*}
where $c'$ and $c''$ are universal.
This proves Step 3.

 \smallskip
\noindent
{\bf Step 4.} We show that for some universal constant $C^\ast$, \[d_h:=\sup_{S_u(y, h)} x_n\leq C^\ast h^{\frac{1}{2-\alpha}} \quad\text{and}\quad |S_u(y, h)| \leq C^\ast |S_u(0, h)|.\] 
By John's lemma, there is an ellipsoid $E$ with center $b$ such that 
\[E-b\subset S_{c_1 h^{\frac{1-\alpha}{2-\alpha}}}' -b\subset n(E-b).\]
In view of Step 3, \[|E| \geq n^{-n}|S_{c_1 h^{\frac{1-\alpha}{2-\alpha}}}'|\geq  c_2 h^{\frac{n-1}{2}} h^{\frac{1}{2-\alpha}},\quad \text{and}\quad
S_{c_1 h^{\frac{1-\alpha}{2-\alpha}}}' \subset S_u(0, h) \subset\{0\leq x_n\leq (c^{-1}h)^{\frac{1}{2-\alpha}}\}.\]
Thus, 
\[\mathcal{H}^{n-1} (E\cap\{x_n= b_n\})\geq c_2 h^{\frac{n-1}{2}},\]
where $c_2$ is universal. 
It follows that
\[|S_u(y, h)| \geq n^{-1} d_h \mathcal{H}^{n-1} (E\cap\{x_n= b_n\})
\geq n^{-1}c_2 d_hh^{\frac{n-1}{2}}.  \]
On the other hand, 
\[\det D^2 u\geq \lambda \dist^{-\alpha} (\cdot,\p\Omega) \geq \lambda d_h^{-\alpha}\quad \text{in } S_u(y, h).\]
By the volume estimate in Lemma \ref{volsec}(ii), we have
\[|S_u(y, h)| \leq C(n) ( \lambda h_h^{-\alpha})^{-1/2} h^{n/2} = C(n)\lambda^{-1/2} h^{n/2} d_h^{\alpha/2}.\]
Consequently,
\[n^{-1}c_2 d_hh^{\frac{n-1}{2}} \leq C(n)\lambda^{-1/2} h^{n/2} d_h^{\alpha/2}.\]
This gives the upper bound for $d_h$ and the inequality for $S_u(y, h)$ as asserted in Step 4.

 \smallskip
\noindent
{\bf Step 5.} If $z\in S_u(y, 3h/4)$, then 
\begin{equation} 
\label{z3h4}
c' h^{\frac{1}{2-\alpha}}\leq \dist(z,\p\Omega) \leq (c^{-1} h)^{\frac{1}{2-\alpha}}.\end{equation}
Indeed, it follows from Steps 3 and 4 that
\[|S'_a| = |S_u(y, h)|  \leq |S'_{C_1 h^{\frac{1-\alpha}{2-\alpha}}}|,\]
so
\[a\leq C_1 h^{\frac{1-\alpha}{2-\alpha}}\quad\text{and}\quad S_u(y, h)\subset S'_{C_1 h^{\frac{1-\alpha}{2-\alpha}}}.\]
Therefore
 \[y_n=\frac{u(y) + h}{a}\geq \frac{h}{a} >C^{-1}_1 h^{\frac{1}{2-\alpha}}.\]
 From $y_n\leq d_h$ and arguing as in Step 3, we also obtain
 \[C^\ast h^{\frac{1}{2-\alpha}}\geq d_h\geq y_n \geq \dist(y,\p\Omega) \geq \bar c  h^{\frac{1}{2-\alpha}}.\]
 Due to Step 4, we only need to prove the lower bound in \eqref{z3h4}. Note that
\[S_u(y, 3h/4)=\{x\in\overline{\Omega}: u(x)< ax_n -h/4\}.
\]
Thus, for $z\in S_u(y, 3h/4)$, we deduce from Step 2 that
\[z_n\geq \frac{h}{4a}\geq c_0 h^{\frac{1}{2-\alpha}}.\]
The rest of the  proof is similar to Step 3.

 \smallskip
\noindent
{\bf Step 6.} We prove that, for some universal constant $C$, 
\begin{equation} 
\label{SuE}
C^{-1} \calE_h \subset S_u(y, h)-y\subset 2C \calE_h.
\end{equation}
Clearly,
 \[S_u(y, h)\subset S_u(0, Ch)\subset C \calE_h.\]
 Hence
  \[S_u(y, h)-y\subset 2C \calE_h\equiv 2CA_h \calE.\]
  
  \smallskip
\noindent 
{\bf Rescaling.} Consider the rescaling $u_h$ of $u$ given by
\begin{equation*}
\label{uh_res}
u_h( x):=h^{-1} [u( y+A_h x)-u(y)-Du(y)\cdot (A_h x)-h],\quad x\in A^{-1}_h (\Omega-y).
\end{equation*}
Then 
\[\det D^2 u_h (x) = h^{-n} (\det A_h)^2 \det D^2 u(y+ A_h x)=  h^{\frac{\alpha}{2-\alpha}} \det D^2 u(y+ A_h x).\]
Note that $x\in S_{u_h}(0, 3/4)$ if and only if $y+ A_h x \in S_u(y, 3h/4)$. Thus, by the distance estimates in Step 5, we can find universal constants $\lambda_0,\Lambda_0$ such that 
\begin{equation}
\label{uh_res}
\lambda_0 \leq \det D^2 u_h\leq \Lambda_0\quad\text{in } S_{u_h}(0, 3/4).
\end{equation}
By the Aleksandrov maximum principle, we have
\[1/4^n = |u_h(0) + 3/4|^n \leq C(n) \dist (0, \p S_{u_h}(0, 3/4)) [\diam (S_{u_h}(0, 3/4))]^{n-1} \Lambda_0 |S_{u_h}(0, 3/4)|.\]
Since $S_{u_h}(0, 1):=A_h^{-1} (S_u(y, h)-y)\subset 2C\calE$, we find a universal constant $c$ such that
 \[\dist (0, \p S_{u_h}(0, 3/4)) \geq c.\]
 Because
   the convex sets $S_{u_h}(0, 3/4)\subset 2C\calE$
have comparable volumes,  $S_{u_h}(0, 3/4)$ must contain $\kappa\calE$ for $\kappa$ universal. 
It follows that \eqref{SuE} holds.

The proposition is proved. 
 \end{proof}

\section{H\"older estimates for the gradient}
\label{SectC1b}
In this section, we prove global H\"older estimates for the gradient of the solution to \eqref{MA1}. Clearly, Theorem \ref{mainthm1} (i) is a consequence of the following result.
\begin{thm}[Global $C^{1,\beta}$ regularity for singular Monge-Amp\`ere equations]
\label{thmb}
Let $\Omega$ be a bounded convex domain in $\R^n$ with $C^2$ boundary $\partial\Omega$. Let $\alpha\in (0, 1)$ and $0<\lambda\leq\Lambda$. Assume $u\in C(\overline{\Omega})$ is a convex function satisfying
\begin{equation*}
 \lambda [\emph{dist}(\cdot,\p\Omega)]^{-\alpha}\le \det D^2 u \le\Lambda [\emph{dist}(\cdot,\p\Omega)]^{-\alpha}\quad\mathrm{in}\;\Omega,
\end{equation*}
 and on $\partial\Omega$, $u$ separates quadratically from its tangent hyperplane, namely,  
there exists $\mu>0$ such that for all $ x_0, x\in \p\Omega$, 
\[\mu|x-x_0|^2\le u(x)-u(x_0)-D u(x_0)\cdot(x-x_0)\le\mu^{-1}|x-x_0|^2.\]
Then, there exist constants $\beta\in (0, 1)$ and $C>0$ depending only on $n$, $\lambda$, $\Lambda$, $\alpha$, $\mu$, 
$\mathrm{diam}(\Omega)$, and the  $C^2$ regularity of $\p\Omega$, 
such that 
\[[Du]_{C^{\beta}(\overline{\Omega})} \leq C.\]
\end{thm}

\begin{proof}
We divide the proof into several steps.
 
 \smallskip
\noindent
{\bf Step 1.} Oscillation estimate for $Du$ in a \index{sections of Aleksandrov solution!maximal interior section}maximal interior section.
For $y\in \Omega$, let $S_u(y, \bar h)$ be the maximal interior section of $u$ centered at $y$, and let  $y_0 \in \p \Omega$ satisfy
\[|y-y_0|=r:=\dist(y,\p\Omega).\] 
We show that, if $r$ is small, universal, then for some universal constant $C_1>0$, and $\alpha_1:= \alpha_0(1+\alpha_0)/2$ where $\alpha_0\in (0, 1-\alpha)$ is universal,
\[|Du(z_1)-Du( z_2)| \leq C_1|z_1-z_2|^{\alpha_0}\quad \text{ in } S_u(y,\bar h/2)\quad\text{ and }\quad |Du(y)-Du(y_0)|\leq r^{\alpha_1}.\]
Indeed,  
if $r\leq c_1$ where $c_1$ is small, universal, then $\bar h\leq c$, and 
by Proposition \ref{tan_sec0} applied at the point $x_0\in \p S_u(y, \bar h) \cap \p \Omega,$ we have 
 $$c\bar h^{\frac{1}{2-\alpha}}\leq r\leq C\bar h^{\frac{1}{2-\alpha}},\quad  |Du(y)-Du(x_0)| \le C \bar h^{\frac{1-\alpha}{2-\alpha}},\quad c \calE_{\bar h}(x_0) \subset S_u(y, \bar h)-y \subset C \calE_{\bar h}(x_0).$$

For simplicity, we can assume $\Omega\subset\R^n_{+}=\{x\in\R^n: x_n>0\}$, $x_0=0\in\p\Omega$,
$u(0)=0$, and $Du(0)=0$. Then $\calE_{\bar h}(x_0) = A_{\bar h} \calE$.
Consider the rescaling $u_{\bar h}$ of $u$ given by
$$u_{\bar h}(\tilde x):=\bar h^{-1} [u( y+A_{\bar{h}}\tilde x)-u(y)-Du(y)\cdot ( A_{\bar{h}}\tilde x)-\bar h].$$
Let 
\[\tilde S_t=A_{\bar h}^{-1}(S_u(y, t \bar h)- y).\]
Then, as in \eqref{uh_res} and the arguments following it, we can find  universal constants $\lambda_0,\Lambda_0$, $c, C$ such that 
\[\lambda_0 \leq \det D^2 u_{\bar h}\leq \Lambda_0\quad\text{in } S_{u_{\bar h}}(0, 3/4),\quad B_c(0) \subset \tilde S_{3/4} \subset B_C(0).\]
By Caffarelli's interior H\"older gradient estimates for the Monge-Amp\`ere equation \cite{Cc1a},  there exist
universal constants $\alpha_0\in (0, 1-\alpha)$ and $C_1$ such that
\begin{equation}
\label{c1aeq}
|D u_{\bar h}(\tilde z_1)-D u_{\bar h}(\tilde z_2)| \le C_1(n,\lambda,\Lambda, \alpha) |\tilde z_1-\tilde z_2|^{\alpha_0} \quad \text{for all }\tilde z_1,\tilde z_2 \in \tilde S_{1/2}.
\end{equation}
Moreover, we have the following size estimate for sections
\begin{equation}
\label{sizeeq}
\tilde S_{\delta}\subset B_{C_1 \delta^{\frac{1}{1+\alpha_0}}}(0)\quad\text{for all } \delta\in (0, 1/2).
\end{equation}

For $\tilde z\in \tilde S_1$, let $z= y+  A_{\bar{h}}\tilde z$. 
Rescaling back the estimate (\ref{c1aeq}), and using
$ \tilde z_1-\tilde z_2=A_{\bar h}^{-1}(z_1-z_2)$, 
we find, for all $z_1, z_2 \in  S_u(y, \bar h/2)$,
\begin{eqnarray*}|Du(z_1)-Du( z_2)|  &=& \big|\bar h (A^t_{\bar h})^{-1}(Du_{\bar h}(\tilde z_1)-Du_{\bar h}(\tilde z_2))\big|\\
 &\leq& C_1  \bar h \|A^{-1}_{\bar h}\|^{1+\alpha_0} |z_1-z_2|^{\alpha_0} \\&\leq&  C_1  \bar h (C\bar h^{-\frac{1}{2-\alpha}})^{1+\alpha_0} |z_1-z_2|^{\alpha_0}\\
 &\leq&  C_1|z_1-z_2|^{\alpha_0},
\end{eqnarray*}
if $c_1$ is small.

From $|y-y_0|=r$, we have
$$|x_0-y_0|\leq |x_0-y|+ |y-y_0| \leq C\bar h^{1/2}  + r \leq C r^{\frac{2-\alpha}{2}},$$
if $c_1$ is small.
 By  Proposition \ref{U-bound-grad2}, we then find
\begin{eqnarray*}|Du(y)-Du(y_0)| &\leq& |Du(y)-Du(x_0)|+|Du(x_0)-Du(y_0)|\\ &\leq& C\bar h^{\frac{1-\alpha}{2-\alpha}} + C|x_0-y_0|^{1-\alpha}\leq  r^{\alpha_1}.  
\end{eqnarray*}

 \smallskip
\noindent
{\bf Step 2.} Oscillation estimate for $Du$ near the boundary. 

Let $x, y\in\Omega$ with $\max\{\dist(x,\p\Omega), \dist(y,\p\Omega), |x-y|\}\leq c_1$ small.  
We show that 
\[|Du(x)-Du(y)|\leq \max\{C_1|x-y|^{\alpha_0}, |x-y|^{\alpha_1}\}.\] 
Indeed, let $x_0, y_0\in\p\Omega$ be such that \[|x-x_0|=\dist(x,\p\Omega):= r_x \quad\text{and}\quad |y-y_0|=\dist(y,\p\Omega):= r_y.\] We can assume $r_y\leq r_x\leq c_1$.

 Note that
$(1/2) (S_u(y,\bar h(y))-y)\subset S_u(y,\bar h(y)/2)-y$.
Thus, for $c_1$ and $c$ small, universal, 
$$S_u(y, \bar h(y)/2)\supset B_{\bar h(y)^{\frac{1}{2-\alpha}}/(2C)}(y)\supset B_{cr_y} (y).$$

 \smallskip
 If $|y-x|\leq c r_x$, then $y\in S_u(x, \bar h(x)/2)$ and hence, 
  Step 1 gives \[|Du(x)-Du(y)|\leq C_1|x-y|^{\alpha_0}.\]

Consider now the case $|y-x|\geq c r_x$. Then,
\[|x_0-y_0|\leq |x_0-x| + |x-y| + |y-y_0|\leq 2 r_x + |x-y| \leq C|x-y|.\]
 Thus, from 
\begin{equation*}
|Du(x)-Du(y)|\leq |Du(x)-Du(x_0)|+ |Du(x_0)- Du(y_0)| + |Du(y_0)-Du(y)|,
\end{equation*}
Step 1 and Proposition \ref{U-bound-grad2}, and noting that $1-\alpha>\alpha_1$, we have
\begin{eqnarray*}
|Du(x)-Du(y)| &\leq& r_x^{\alpha_1} + C|x_0- y_0|^{1-\alpha}  + r_y^{\alpha_1}
\\ &\leq & 2r_x^{\alpha_1} + C|x-y|^{1-\alpha} \leq |x-y|^{\alpha_1}.
\end{eqnarray*}

 \smallskip
\noindent
{\bf Step 3.} Conclusion. 
 By the convexity of $u$, we  have $\text{osc}_{\Omega}|Du|\leq \text{osc}_{\p\Omega}|Du|\leq C(n,\alpha,\Omega)$.
Combining this with Step 2 and the interior H\"older gradient estimates, we easily obtain the conclusion of the theorem.
 \end{proof}

 \begin{rem}
\label{sizeS}
Assume that $u$ and $\Omega$ satisfy the hypotheses of  Theorem \ref{SZthm}.
Let $y\in\Omega$ and let $x_0\in \p S_u(y, \bar h(y)) \cap \p \Omega$. Then, the size estimate \eqref{sizeeq} implies that in the orthogonal coordinate frame $(\tau_{x_0}, -\nu_{x_0})$, for any $\delta\in (0, 1/2)$, $S_u(y,\delta \bar h(y))-y$ is contained in the rectangular box centered at the origin with size lengths
$C \delta^{\frac{1}{1+\alpha_0}}(\bar h^{\frac{1}{2}}, \cdots, \bar h^{\frac{1}{2}}, \bar h^{\frac{1}{2-\alpha}})$, that is, \[S_u(y, \delta\bar h(y))-y\subset  C\delta^{\frac{1}{1+\alpha_0}} A_{\bar h(y)} B_1(0).\] 
\end{rem}
 \section{Vitali covering lemma} 
 \label{SectV}
 In this section, we prove a Vitali-type covering lemma for sections that will be used in the global $W^{2, p}$ estimates. Since $W^{2, p}$ estimates are standard in the interior, we only consider sections 
whose concentric maximal interior sections have small heights.

Recall that 
\[\Omega_c:=\{x\in\Omega: \dist(x,\p\Omega)<c\},\]
where we take $c$ to be small, universal.
 
For nonsingular Monge-Amp\`ere equations, Vitali-type covering lemmas follow from standard arguments using the engulfing properties of sections. Instead of establishing these properties for our singular equation \eqref{MA1}, we will use the following observation.
\begin{lem}
\label{h12lem}
Assume that $u$ and $\Omega$ satisfy the hypotheses of  Theorem \ref{SZthm}. Then,
there exist a universal constant
$\delta\in (0,1/4)$  with the following property. If $x_1, x_2\in\Omega_c$,   $S_u(x_1, \delta \bar h(x_1))\cap S_u(x_2, \bar h_2(x_2))\neq\emptyset$, and $2\bar h(x_1)\geq \bar h(x_2)$, 
   then \[S_u(x_2, \delta \bar h(x_2))\subset
S_u(x_1, \bar h(x_1)/2).\] 
\end{lem}

\begin{proof}
Let 
\[\bar\delta:=\delta^{\frac{1}{1+\alpha_0}}, \quad h_1:=\bar h (x_1),\quad h_2:= \bar h(x_2),\quad z_1=\p S_u(x_1, h_1)\cap \p\Omega, \quad  z_2=\p S_u(x_2, h_2)\cap \p\Omega.\]
By Remark \ref{sizeS}, in the orthogonal coordinate frame $(\tau_{z_i}, -\nu_{z_i})$, 
\begin{equation}
\label{sizedelh}
S_u(x_i, \delta h_i)-x_i\subset  C\bar \delta A_{h_i} B_1(0) \quad\text{for } i=1, 2.\end{equation}
Assume $2h_1\geq h_2$. Since $S_u(x_1, \delta  h_1)\cap S_u(x_2, h_2)\neq\emptyset$, the triangle inequality gives
\begin{equation}
\label{x12ineq}
|x_1-x_2| \leq \bar\delta (c^{-1} h_1)^{\frac{1}{2}} +  \bar\delta (c^{-1} h_2)^{\frac{1}{2}} \leq C \bar\delta h_1^{\frac{1}{2}}. \end{equation}
By Proposition \ref{tan_sec0}, there exists a universal constant $\kappa$ such that 
 \begin{equation}
 \label{sizeh12}
 \kappa \calE_{ h_i}(z_i) \subset S_{u}(x_i,  h_i/2) -x_i\ \subset S_{u}(x_i,  h_i) -x_i\subset \kappa^{-1} \calE_{h_i}(z_i)\quad\text{for } i=1, 2.\end{equation}
 We have
 \[|z_1-z_2| \leq |z_1-x_1| + |x_1-x_2| + |x_2-z_2| \leq C h_1^{\frac{1}{2}}.\]
 Since $\p\Omega$ is $C^2$, we have
 \begin{equation} 
 \label{nuz12}
 |\nu_{z_1}- \nu_{z_2}|\leq C|z_1-z_2| \leq C h_1^{\frac{1}{2}}.\end{equation}
 Now, let $y\in S_u(x_2, \delta h_2)$. In view of \eqref{sizeh12}, 
 we show that $y\in S_u(x_1, h_1/2)$ for $\delta$ small, universal by establishing that
  \begin{equation}
  \label{tanineq}
  |(y-x_1)\cdot \tau_{z_1}| \leq \kappa h_1^{\frac{1}{2}}\end{equation}
  and
 \begin{equation}
 \label{nuineq}
 |(y-x_1)\cdot \nu_{z_1}| \leq \kappa h_1^{\frac{1}{2-\alpha}}.\end{equation}
 For the tangential components, we write
 \[(y-x_1)\cdot \tau_{z_1} = (y-x_2)\cdot \tau_{z_1} + (x_2-x_1)\cdot \tau_{z_1}.\]
 Then, recalling \eqref{sizedelh} and \eqref{x12ineq}, we obtain 
 \[|(y-x_1)\cdot \nu_{z_1}| \leq C \bar\delta h_1^{\frac{1}{2}}+ C \bar\delta h_1^{\frac{1}{2}}  \leq \kappa h_1^{\frac{1}{2}},\]
 if $\bar \delta$ is small. Since $\bar \delta = \delta^{\frac{1}{1+\alpha_0}}$, this is the case when $\delta$ is small. 
 
 For the normal component, choose $z\in S_u(x_1, \delta  h_1)\cap S_u(x_2, h_2)$ and we write
 \[ (y-x_1)\cdot \nu_{z_1} = (y-z)\cdot (\nu_{z_1}-\nu_{z_2}) + (y-x_2)\cdot \nu_{z_2} + (x_2-z)\cdot \nu_{z_2} + (z-x_1)\cdot \nu_{z_1}.\]
 In view of \eqref{sizedelh},  \eqref{nuz12}, we find
 \[| (y-x_1)\cdot \nu_{z_1} | \leq (C\bar\delta   h_1^{\frac{1}{2}})(C h_1^{\frac{1}{2}}) + C\bar \delta h_1^{\frac{1}{2-\alpha}} \leq \kappa h_1^{\frac{1}{2-\alpha}},\]
 if $\bar \delta$ and $h_1\leq c$ where $c$ is small. The proof of the lemma is complete.
\end{proof}

 \begin{lem}[Vitali covering lemma] 
 \label{V_lem}
Assume that $u$ and $\Omega$ satisfy the hypotheses of  Theorem \ref{SZthm}.
Then, there exist a universal constant
$\delta\in (0,1/4)$ 
and
 a countable subcollection of disjoint sections 
 $\displaystyle \{S_u(x_i, \delta h(x_i))\}_{i=1}^{\infty}$, where $x_i\in\Omega_c$, such that
 $$\displaystyle \Omega_c\subset \bigcup_{i=1}^\infty S_u(x_i, \bar h(x_i)/2).$$
\end{lem}
\begin{proof} Let $\delta$ be as in Lemma \ref{h12lem}.
Let $\mathcal{S}$
be the collection of sections $S^x= S_u(x, h (x))$, where $x\in\Omega_c$ and $h(x)=\delta \bar h(x)$.
Let 
$d(\mathcal{S}):= \sup\{h(x): S^x\in\mathcal{S}\}\leq c.$
Define $$\mathcal{S}_i\equiv \{S^x\in \mathcal{S}: \frac{d(\mathcal{S})}{2^i}<h(x) \leq \frac{d(\mathcal{S})}{2^{i-1}}\}~(i=1,2,\ldots),$$ and $\mathcal{F}_i\subset \mathcal{S}_i$ as follows. 
Let $\mathcal{F}_1$ be a maximal disjoint collection of sections in $\mathcal{S}_1$. By the volume estimate, $\mathcal{F}_1$ is finite. 
Assuming $\mathcal{F}_1,\ldots, \mathcal{F}_{i-1} $ have been selected, we choose $\mathcal{F}_i$
to be any maximal disjoint subcollection of
$$\Bigg\{S\in \mathcal{S}_i: S\cap S^{x}=\emptyset~\text{for all~} S^x\in \bigcup_{j=1}^{i-1}\mathcal{F}_j\Bigg\}.$$ Again, each $\mathcal{F}_i$ is a finite set. Let
$\mathcal{F}:=\bigcup_{k=1}^{\infty} \mathcal{F}_i$,
and consider the countable subcollection of disjoint sections $S_u(x_i, h(x_i))$ where $S^{x_i}\in \mathcal{F}$. 

We now show that this subcollection
 satisfies the conclusion of the lemma. Indeed, let $S^x$ be any section in $\mathcal{S}$. Then,  there is an index $j$ such that $S^x\in \mathcal{S}_j$. By the maximality of $\mathcal{F}_j$, there is a section $S^y\in \bigcup_{i=1}^j 
\mathcal{F}_i$ with $S^x\cap S^y \neq\emptyset$. Note that
$h(x) \leq 2 h(y)$
because $h(y)> \frac{d(\mathcal{S})}{2^j} $ and $h(x) \leq \frac{d(\mathcal{S})}{2^{j-1}}$. Thus,
by the choice of $\delta$,  \[S^x=S_u(x, \delta \bar h(x))\subset S_u(y,  \bar h(y)/2)\subset \bigcup_{i=1}^\infty S_u(x_i, \bar h(x_i)/2).\]
The lemma is proved.
\end{proof}

\section{Global second order derivative estimates}
\label{SectW2p}
In this section, we establish global $W^{2,p}$ estimates for the solution to \eqref{MA1}. Clearly, Theorem \ref{mainthm1}(ii) is a consequence of Theorem \ref{mainthm1}(i) and the following result.
\begin{thm}[Global $W^{2,p}$ estimates for singular Monge-Amp\`ere equations]
\label{thmw}
Let $\Omega$ be a bounded convex domain in $\R^n$ with $C^2$ boundary $\partial\Omega$. Let $\alpha\in (0, 1)$ and $0<\lambda\leq\Lambda$. Assume $u\in C(\overline{\Omega})$ is a convex function satisfying
\begin{equation*}
  \det D^2 u= g  [\emph{dist}(\cdot,\p\Omega)]^{-\alpha} \quad\mathrm{in}\;\Omega, \quad g\in C(\overline{\Omega}), \quad \lambda \leq g\leq \Lambda,
\end{equation*}
 and on $\partial\Omega$, $u$ separates quadratically from its tangent hyperplane, namely,  
there exists $\mu>0$ such that for all $ x_0, x\in \p\Omega$, 
\[\mu|x-x_0|^2\le u(x)-u(x_0)-D u(x_0)\cdot(x-x_0)\le\mu^{-1}|x-x_0|^2.\]
Then, for each $p<1/\alpha$, there exists a constant $C>0$ depending only on $n$, $p$, $\lambda$, $\Lambda$, $\alpha$, $\mu$, the modulus of continuity of $g$ in $\overline\Omega$, 
$\mathrm{diam}(\Omega)$, and the  $C^2$ regularity of $\p\Omega$, 
such that 
\[\|D^2 u\|_{L^p(\Omega)} \leq C.\]
\end{thm}

\begin{proof}
As in the global $W^{2, p}$ estimates for nonsingular Monge-Amp\`ere equations in \cite{Sw2p},  we divide the proof into several steps.

\noindent
{\bf Step 1.} $L^p$ estimate for the Hessian in \index{sections of Aleksandrov solution!maximal interior section}maximal interior sections. Consider the maximal interior section $S_{u}(y, h)$ of $u$ centered at $y\in\Omega$ with height $h:= \bar{h}(y)$.
Let $\bar y =\p S_u(y, h)\cap\p\Omega$. 
If $h\leq c$ where $c$ is small, universal, then Proposition~\ref{tan_sec0}   gives
\[\kappa_0 \calE_h(\bar y) \subset S_{u}(y, h)-y \subset \kappa_{0}^{-1} \calE_h(\bar y).\]
For simplicity, we can assume $\Omega\subset\R^n_{+}=\{x\in\R^n: x_n>0\}$, $\bar y=0\in\p\Omega$,
$u(0)=0$, and $Du(0)=0$. Then $\calE_{h}(\bar y) = A_{h} \calE$. Let $1<p<\infty$ to be chosen later.

We use the rescalings:
\[u_{h}(x) := 
h^{-1} \big[u ( y +  A_{h} x)-u(y)-Du(y)\cdot(A_h x) - h\big],\]
for $x\in\Omega_h:= A^{-1}_{h} (\Omega-y)$. 
Then
$$B_{\kappa_0}(0)\subset S_{u_{h}} (0, 1)\equiv A^{-1}_{h} \big(S_{u}(y, h) -y\big)\subset B_{\kappa_0^{-1}}(0).$$
We have
\begin{eqnarray*}\det D^2 u_h(x) 
&=& h^{-n} (\det A_h)^2 \det D^2 u(y+ A_h x)  \\
&=& g(y+ A_h x) \Big[h^{\frac{-1}{2-\alpha}} \dist(y+ A_h x,\p\Omega)\Big]^{-\alpha}
\equiv f(x).\end{eqnarray*}
When $x\in S_{u_{h}} (0, 3/4)$, Proposition \ref{tan_sec0} gives 
\[\kappa_0\leq h^{\frac{-1}{2-\alpha}} \dist(y+ A_h x,\p\Omega)\leq \kappa_0^{-1}.\]
Since the map $x\in S_{u_{h}} (0, 3/4)\mapsto h^{\frac{-1}{2-\alpha}} \dist(y+ A_h x,\p\Omega)$ is Lipschitz with Lipschitz norm bounded by $1$, we deduce that 
$$f\in C(\overline{S_{u_{h}} (0, 3/4)}),\quad \lambda_0\leq f= \det D^{2} u_{h} \leq\Lambda_0 \quad\text{in } S_{u_{h}} (0, 3/4), 
\quad u_{h} =0 \, \mbox{ on }\, \partial  S_{u_{h}} (0, 1).$$
By Caffarelli's interior $W^{2, p}$ estimates \cite{Cw2p}, there is constant $C(p)$ depending only on $p, n,\alpha,\mu$, $g$, and $\Omega$ such that 
\[\int_{S_{u_{h}} (0, 1/2)}|D^2  u_h|^ p\, dx\leq C(p).\]
Since
$D^{2} u( y +  A_{h} x) = h (A_{h}^{-1})^{t} \, D^{2} u_{h}(x) \, A_{h}^{-1}$,
we obtain
\begin{equation}
\label{w2p-small-local}
\begin{split}
\int_{S_{u}(y,h/2)}|D^{2} u(z)|^{p} \, dz&= h^{p} \det A_h 
\int_{S_{u_{h}} (0, 1/2)} |(A_{h}^{-1})^{t} \, D^{2} u_{h}(x) \, A_{h}^{-1}|^{p}\, dx
\\ &\leq  C(p) h^{\frac{n-1}{2} + \frac{1}{2-\alpha} -\frac{p\alpha}{2-\alpha}} \int_{S_{u_{h}} (0, 1/2)} |D^{2} u_{h}(x)|^{p}\, dx \\&\leq C(p) h^{\frac{n-1}{2} + \frac{1-p\alpha}{2-\alpha}}.
\end{split}
\end{equation}

From  Proposition~\ref{tan_sec0}, we find  that 
if  $y\in \Omega$ with $\bar{h}(y)\leq c$ small,   then
$$S_{u}(y, \bar{h}(y))\subset (y + \kappa_{0}^{-1}\calE_h)\cap\overline{\Omega}\subset \Omega_{C\bar{h}(y)^{\frac{1}{2-\alpha}}}: =\big\{x\in \overline{\Omega}: \, 
\dist(x, \partial\Omega)< C\bar{h}(y)^{\frac{1}{2-\alpha}}\big\}.$$
We can reduce $c$ so that 
\[\bar h(y)\leq c\quad\text{in } \Omega_c.\]
By Caffarelli's interior $W^{2, p}$ estimates \cite{Cw2p}, we have
\begin{equation}
\label{W2pfar}
\int_{\Omega\setminus \Omega_c} |D^{2} u|^{p}\, dx\leq C(p).\end{equation}
It remains to consider $W^{2,p}$ estimates near the boundary.

\smallskip
\noindent
{\bf Step 2.} A covering argument. By the Vitali covering Lemma (Lemma \ref{V_lem}), there exists a covering $\cup_{i=1}^{\infty} S_{u}(y_{i}, \bh(y_i)/2)$ of $\Omega_c$ where the sections $S_{u}(y_{i}, \delta\bh(y_i))$ with $y_i\in\Omega_c$ are disjoint for some 
universal
$\delta\in (0, 1/2)$. We have
\begin{equation}
\label{Lphh}
\int_{\Omega_c} |D^{2} u|^{p}\, dx \leq \sum_{i=1}^\infty\int_{S_{u}(y_{i}, \bh(y_i)/2)} |D^2 u|^{p}\, dx.
\end{equation}
We will estimate the sum 
in
(\ref{Lphh}), depending on the heights $\bar h(y_i)$. Note that, by Proposition \ref{tan_sec0}, there exists a  universal constant 
$c_0>0$ such that
\begin{equation}
\label{w21_vol}
|S_{u}(y_{i}, \delta\bar{h}(y_i))| \geq c_0 \bar{h}(y_i)^{\frac{n-1}{2} + \frac{1}{2-\alpha}}.
\end{equation}

 \smallskip
\noindent
For
$d\leq c$, we consider the family $\mathcal{F}_{d}$ of indices $i$ for sections $S_{u}(y_{i}, \bar{h}(y_i)/2)$ such that \[d/2<\bar{h}(y_i)\leq d.\] Let $M_{d}$ be the number 
of indices in $\mathcal{F}_{d}$. Since $S_{u}(y_{i}, \delta\bar{h}(y_i))\subset \Omega_{Cd^{\frac{1}{2-\alpha}}}$ are disjoint for $i\in \mathcal{F}_{d}$, we find from (\ref{w21_vol}) that
\begin{eqnarray*} M_d c_0  (d/2)^{\frac{n-1}{2} + \frac{1}{2-\alpha}}&\leq& \sum_{i\in\mathcal{F}_{d}} |S_{u}(y_{i},\delta \bar{h}(y_i))| \\ &\leq& |\Omega_{Cd^{\frac{1}{2-\alpha}}}|\leq C_{\ast}d^{\frac{1}{2-\alpha}},
\end{eqnarray*}
where $C_{\ast}>0$ depends only on $n$ and $\Omega$. Therefore
\begin{equation}\label{Md-est}
 M_{d}\leq C_{b}d^{-\frac{n-1}{2}}.
\end{equation}

\noindent
It follows from 
(\ref{w2p-small-local}) and \eqref{Md-est} that
\begin{eqnarray*}
\sum_{i\in \mathcal{F}_{d}} \int_{S_{u}(y_{i}, \bh(y_i)/2)} |D^{2} u|^{p}\, dx &\leq& C(p) M_{d} d^{\frac{n-1}{2} + \frac{1-p\alpha}{2-\alpha}}\\ &\leq&
C(p) d^{ \frac{1-p\alpha}{2-\alpha}}. 
\end{eqnarray*}
Adding these inequalities for $d = c2^{-k}$ where $k= 0, 1, 2,...,$  
we obtain 
\begin{equation}
\label{W2pnear}
\begin{split}
\sum_{i=1}^{\infty}\int_{S_{u}(y_{i}, \bh(y_i)/2)} |D^2 u|^{p}\, dx &= \sum_{k=0}^{\infty}\sum_{i\in \mathcal{F}_{c2^{-k}}} \int_{S_{u}(y_{i}, \bh(y_i)/2)} |D^{2} u|^{p}\,dx
\\&\leq \sum_{k=0}^{\infty} C(p) (c2^{-k})^{ \frac{1-p\alpha}{2-\alpha}}\\ &\leq C(n, \alpha, p,\mu, g, \Omega),
\end{split}
\end{equation}
if \[p<1/\alpha.\]
Combining \eqref{W2pfar}, \eqref{Lphh}, and \eqref{W2pnear}, we obtain 
the desired global $L^p$ estimate for $D^2 u$.
\end{proof}

\begin{rem} 
\label{w2pg_rem}
Assume that $u$ and $\Omega$ satisfy the hypotheses of  Theorem \ref{thmw}.
Given $0<p<1/\alpha$, we can show that for all $\gamma\in [0, 1-p\alpha)$,
\[\int_\Omega \dist^{-\gamma}(\cdot,\p\Omega)\|D^2 u\|^p\,dx\leq C(n, \alpha, p, \gamma, \mu, g, \Omega).\]
Indeed, if  $x\in S_{u}(y_{i}, \bh(y_i)/2)$ where $i\in \mathcal{F}_{c2^{-k}}$, then Proposition \ref{tan_sec0} gives
\[\dist (x,\p\Omega)\geq c (c2^{-k})^{\frac{1}{2-\alpha}}.\]
Therefore, for $\gamma\in (0, 1-p\alpha)$, by revisiting \eqref{W2pnear}, we find
\begin{equation*}
\begin{split}
\sum_{i=1}^{\infty}\int_{S_{u}(y_{i}, \bh(y_i)/2)} \dist^{-\gamma}(\cdot,\p\Omega) |D^2 u|^{p}\, dx &= \sum_{k=0}^{\infty}\sum_{i\in \mathcal{F}_{c2^{-k}}} \int_{S_{u}(y_{i}, \bh(y_i)/2)} \dist^{-\gamma}(\cdot,\p\Omega) |D^{2} u|^{p}\,dx
\\&\leq \sum_{k=0}^{\infty} C(p) (c2^{-k})^{\frac{-\gamma}{2-\alpha}} (c2^{-k})^{ \frac{1-p\alpha}{2-\alpha}}\\ &\leq C(n, \alpha, p,\gamma, \mu, g, \Omega).
\end{split}
\end{equation*} 
\end{rem}
Consequently, we obtain the following result from Theorem \ref{mainthm1} and Remark \ref{w2pg_rem}.
\begin{cor}
Let $u\in C(\overline{\Omega})$ be the convex solution to \eqref{MA1} where  $\Omega\subset\R^n$ is a uniformly convex domain with $C^3$ boundary.
Given $0<p<1/\alpha$, we have for all $\gamma\in [0, 1-p\alpha)$,
\[\int_\Omega \dist^{-\gamma}(\cdot,\p\Omega)\|D^2 u\|^p\,dx\leq C(n, \alpha, p, \gamma, \Omega).\]
\end{cor}
\begin{rem} The range of $p$ in Theorem \ref{mainthm1}(ii) is sharp. Consider for example \[\Omega= B_1(0)\subset\R^n.\] Then the solution $u$ to \eqref{MA1} is radial. Thus $u(x) = v(|x|)$ where $v:[0, 1]\rightarrow (-\infty, 0]$ is of class $C^{1,\beta}$ with $\beta =\beta (n,\alpha)>0$ and 
\begin{equation}
\label{MAr}
 \left\{
 \begin{alignedat}{2}
   v'' (r) \Big( \frac{v'(r)}{r}\Big)^{n-1}~&= |v(r)|^{-\alpha} \h~&&\text{in} ~[0, 1), \\\
v(1) &=0,\\\ v'(0)&=0.
 \end{alignedat}
 \right.
\end{equation}
Moreover, there exist positive constant $\lambda$ and $\Lambda$, depending only on $n$ and $\alpha$, such that 
\[\lambda (1-r)\leq |v(r)|\leq \Lambda (1-r)\quad\text{in } [0, 1] \quad\text{and } \lambda \leq v'(1)\leq \Lambda.\]
It follows from \eqref{MAr} and the global $C^{1,\beta}$ regularity of $v$ that 
\[v''(r) \geq c(n,\alpha) (1-r)^{-\alpha}\quad\text{for all } r\in [1/2, 1]\quad \text{where } c(n,\alpha)>0.\]
This implies that $v''\not\in L^{\frac{1}{\alpha}}\big([1/2, 1]\big)$. Since 
\[\|D^2 u\| \geq \frac{1}{n}\Delta u = \frac{1}{n} \big (v'' + \frac{n-1}{r} v'\big)\geq \frac{v''}{n},\]
we find that $D^2 u \not\in L^{\frac{1}{\alpha}} (B_1(0))$.
\end{rem}
\medskip
{\bf Acknowledgements.} The authors would like to thank the referee for carefully reading the paper and providing constructive comments.

\end{document}